\documentclass[pa4paper]{article}
\usepackage{mathrsfs}
\usepackage{amsfonts}
\usepackage[dvipdfm, margin=3cm]{geometry}
\usepackage{amsmath}
\usepackage{amssymb}
\usepackage{amsthm}
\usepackage{latexsym}
\allowdisplaybreaks[4]

\newtheorem{thm}{Theorem}[section]

\newtheorem{lem}[thm]{Lemma}

\newtheorem{rem}[thm]{Remark}

\DeclareMathOperator{\ch}{ch}
\DeclareMathOperator{\End}{End}

\DeclareMathOperator{\Hom}{Hom}
\DeclareMathOperator{\wt}{wt}
\DeclareMathOperator{\om}{\omega}

\DeclareMathOperator{\ind}{lnd}
\DeclareMathOperator{\Res}{Res}
\DeclareMathOperator{\Span}{span}

\def\ha{\frac{1}{2}}
\def\se{\frac{1}{16}}

\newcommand{\BF}{\mathbb{F}}

\newcommand{\BC}{\mathbb{C}}

\newcommand{\BZ}{\mathbb{Z}}
\newcommand{\Z}{\mathbb{Z}}

\def \1{{\bf 1}}
\def \<{\langle}
\def \>{\rangle}
\begin{document}
\title{\bf Modular $A_n(V)$ theory}
\author{Li Ren\footnote{Supported by China NSF grants 11301356 and 11671277}\\
{\small \textit{School of Mathematics, Sichuan University, Chengdu
610064, China}}}
\date{}
\maketitle

\begin{abstract}
A series of  associative algebras $A_n(V)$ for a vertex operator algebra $V$ over an
arbitrary algebraically closed field and  nonnegative integers $n$ are constructed 
such that there is a one to one correspondence between irreducible
$A_n(V)$-modules which are not $A_{n-1}(V)$ modules  and irreducible 
$V$-modules. Moreover,  $V$ is rational  if and only if $A_n(V)$ is semisimple  for all $n.$ In particular, the homogeneous subspaces of any irreducible $V$-module are finite dimensional for rational vertex operator algebra $V.$
\end{abstract}

\section{Introduction}

This paper is an extension of the $A_n(V)$-theory for a vertex operator algebra $V$ from the field $\BC$ of complex numbers \cite{DLM3} to any algebraically closed field $\BF.$
The $A_0(V)=A(V)$ for any field $\BF$ has been investigated previously in \cite{DR1}.

The associative algebra $A(V)$ associated to any vertex operator algebra  was introduced and studied in \cite{Z} over $\BC.$ In the representation theory, the most interesting $V$-modules are
the so called admissible modules \cite{DLM2} which are $\Z_+$-graded modules where $\Z_+$ is the set of nonnegative integers. The $V$ is called rational if the admissible module category is semisimple.
The importance of $A(V)$-theory is that it gives a kind of highest weight representations for vertex operator algebra
without usual triangular decomposition. Given an $A(V)$-module one can construct Verma type admissible $V$-module $M(U)$ such that the top level $M(U)(0)$ of $M(U)$ is $U$ \cite{Z}. Moreover, there is a one to one correspondence between irreducible $A(V)$-modules and irreducible admissible  $V$-modules. So the $A(V)$-theory provides a powerful tool for classification of irreducible admissible $V$-modules. Also see
the $A(V)$-theory for vertex operator superalgebra in \cite{KW} and the twisted representations for vertex operator algebras in \cite{DLM2} and \cite{DZ}.

The $\Z_+$-gradation of an admissible module also leads naturally to the construction of a series of associative algebras $A_n(V)$ over $\BC$ \cite{DLM3} such
that the first $n+1$ homogeneous subspaces of an admissible module are $A_n(V)$-modules. Moreover, $V$ is rational if and only if $A_n(V)$ are finite dimensional semisimple associative algebras for all $n\geq 0$ \cite{DLM3}.
So the associative algebras $A_n(V)$ for all $n$ determine the representation theory for vertex operator algebra $V$ completely in some sense. It has been expected for a long time that the semisimplicity of $A(V)$ is equivalent to the rationality of $V.$ But a proof of this conjecture is not visible at this point. There are also twisted analogues of $A_n(V)$ \cite{DLM4}, \cite{MT} in the study of orbifold theory.

On the other hand, the study of vertex operator algebra over an arbitrary algebraically closed field is very limited except for vertex operator algebra associated to the highest weight module for the Virasoro algebra with central charge $1/2$ \cite{DR2}, lattice vertex operator algebras \cite{M}, Heisenberg vertex operator algebras \cite{LM} and modular moonshine \cite{B2}, \cite{BR}, \cite{GL}. Integral forms of vertex operator algebras studied in
\cite{DG1}, \cite{DG2}, \cite{R1}, \cite{R2} are also useful in constructing modular vertex operator algebras from vertex operator algebras over $\BC.$

The $A(V)$ theory for an arbitrary vertex operator algebra over any algebraically closed field $\BF$ was given in \cite{DR1}.  Almost all results on $A(V)$ in the case
of complex field are still valid. So $A(V)$-theory is still a very powerful tool in the study of representation theory for modular vertex operator algebra. Motivated by the $A_n(V)$-theory developed in \cite{DLM3},
 we construct and study associative algebras $A_n(V)$ for $n\geq 0$ over any algebraically closed field in this paper. Here are our main results: (1) The subspace $M(m)$
of any $\Z_+$-graded  $V$-module $M=\oplus_{m\geq 0}M(m)$ is an
$A_n(V)$-module for $m\leq n.$  (2) Given an $A_n(V)$-module $U$ which is not an $A_{n-1}(V)$-module, one can construct a $V$-module $M(U)$ of Verma type such that
$M(U)(n)=U.$ (3) Sending a $V$-module $M=\oplus_{m\geq
0}M(m)$ with $M(0)\ne 0,$ $M(n)\ne 0$  to $M(n)$ gives a bijection
between irreducible  $V$-modules with $M(n)\ne 0$ and irreducible
$A_n(V)$-modules which are not $A_{n-1}(V)$-modules. (4) $M=\oplus_{n\geq 0}M(n)$ is irreducible $V$-module if and only if each $M(n)$ is an irreducible $A_n(V)$-module. (5) $V$ is rational if and only if $A_n(V)$ is a finite dimensional semisimple associative algebra for all $n\geq 0.$ In particular, the homogeneous subspace $M(n)$ of an irreducible $V$-module is finite dimensional for all $n\geq 0$ when $V$ is rational.

It is worthy to point out that  we could not prove the finite dimension property of an irreducible module in \cite{DR1} for a rational vertex operator algebra $V$ as stated in (5) by using $A(V)$ only. The proof in the case
of complex field \cite{DLM2} uses the operator $\frac{L(1)^m}{m!}$ on $V$ for any $m\geq 0.$ But we do not know  if  $\frac{L(1)^m}{m!}$ exists when the characteristic of $\BF$ is finite.
Proving that $M(n)$ are finite dimensional for irreducible $V$-module $M,$ in fact, is one of our motivations to study the $A_n(V)$ theory for any $n\geq 0.$ We also establish that if $A_n(V)$ is finite dimensional semisimple associative algebra for all $n\geq 0$ then $V$ is rational for any algebraically closed field $\BF.$ This is true
in the case of the complex field \cite{DLM3}. A crucial step in \cite{DLM3} the proof is that $M(n)\ne 0$ if $n$ is large enough for any irreducible module $M.$ We also obtain  this property
over a field of finite characteristic with a different proof as the old proof is not valid anymore.

Our treatment in this paper largely follows \cite{DLM3} with suitable modification to deal with the finite characteristic of the field. Sometimes we have to find a new proof of a result as the proof given in \cite{DLM2} does not work in the current situation. For example, the proof of inequivalence of $A_n(V)$-modules $M(i)$ and $M(j)$ for $i,j\leq n$ and $i\ne j$ in \cite{DLM3} is easy as $L(0)$ has different eigenvalues on $M(i)$ and
$M(j)$ where $M$ is an irreducible $V$-module. But for an arbitrary field $\BF,$ the eigenvalues of $L(0)$ on $M(i)$ and $M(i+pm)$ are always the same where $p$ is the characteristic of $\BF$ and $m$ is any integer.

This paper is organized as follows: We review various notions of modules
for a vertex operator algebra over an algebraically closed field $\BF$ in Section 2. We also discuss the rationality and some consequences of the Jacobi identity.
Section 3 is devoted to the construction of $A_n(V).$ We also investigate important properties of $A_n(V)$ and explain how to get $A_n(V)$-module from a $V$-module. We give a construction of
$V$-module $M(U)$ from an $A_n(V)$-module $U$ which is not an $A_{n-1}(V)$-module such that $M(U)(0)\ne 0$ and $M(U)(n)=U.$ Moreover, $M(U)$ has the largest submodule whose intersection with $M(U)(n)$ is $0$
and the corresponding quotient $L(U)$ of $M(U)$ also satisfies $L(U)(0)\ne 0$ and $L(U)(n)=U.$ One can easily see that $U$ is irreducible $A_n(V)$-module if and only if $L(U)$ is irreducible $V$-module. Applying our results to rational vertex operator algebra, we see that $M(U)=L(U)$ and each $A_n(V)$ is finite dimensional semisimple associative algebra.

The author thanks Professor Chongying Dong for many helpful discussions and valuable comments. Part of this work was done when the author visited the University of California at Santa Cruz. The author also thanks
the Department of Mathematics there for their hospitality.

\section{Basics}
\def\theequation{2.\arabic{equation}}
\setcounter{equation}{0}
In this section we define vertex algebras, vertex
operator algebras and their modules
(cf. \cite{B1}, \cite{FLM}, \cite{LL}, \cite{DR2}) over an algebraically closed
field $\BF$ with $\ch \BF\ne 2.$

A vertex algebra $V=(V,Y,\1)$ over $\BF$ is a vector space equipped
with a linear map
\begin{equation*}
\begin{split}
V& \to(\End V)[[z,z^{-1}]] ,\\
v& \mapsto Y(v,z)=\sum_{n\in{\BZ}}v_nz^{-n-1} \quad (v_n\in\End V)
\end{split}
\end{equation*}
and with a distinguished vector ${\bf 1}\in V$, satisfying the
following conditions for $u, v \in V$, and $m,n\in\BZ:$
\begin{align*}
&   u_nv=0 \text{ for } n \text{ sufficiently large};  \\
& Y({\bf 1},z)=Id_{V}; \\
& Y(v,z){\bf 1}\in V[[z]]\text{ and } \lim_{z\to0}Y(v,z){\bf 1}=v;\\
\end{align*}
and the Jacobi identity holds:
\begin{equation*}
\begin{split}
z^{-1}_0\delta\left(\frac{z_1-z_2}{z_0}\right)&Y(u,z_1)Y(v,z_2)-z^{-1}_0\delta\left(\frac{z_2-z_1}{-z_0}\right)Y(v,z_2)Y(u,z_1)\\
&=z_2^{-1}\delta\left(\frac{z_1-z_0}{z_2}\right)Y(Y(u,z_0)v,z_2).\end{split}
\end{equation*}

Following \cite{B1} we define operators $D^{(i)}$ for $i\geq 0$ on vertex algebra $V:$
$D^{(i)}: V\to V$ such that $D^{(i)}v=v_{-i-1}\1.$ In fact, $D^{(i)}$ is $\frac{D^i}{i!}$ when $\BF=\BC$
where $D=D^{(1)}.$
Set $e^{zD}=\sum_{i\geq 0}D^{(i)}z^i.$ We still  have the skew symmetry $Y(u,z)v=e^{zD}Y(v,-z)u$ for $u,v\in V.$
Let ${\cal D}V=\sum_{i>0}D^{(i)}V.$ Then for any vertex algebra $V,$ $V/{\cal D}V$ is a Lie
algebra such that
$[u,v]=u_0v$ for $u,v\in V.$

A vertex operator algebra $V=(V,Y,\1,\omega)$ over $\BF$ is a
$\BZ$-graded vertex algebra $(V,Y,\1)$
\[
V=\bigoplus_{n\in{ \BZ}}V_n
\]
satisfying $\dim V_{n}< \infty$ for all $n$ and $V_m=0$ if $m$ is
sufficiently small,
with a Virasoro vector $\omega\in V_2$ such that the following
conditions hold for $u, v \in V$, and $m,n\in\BZ$, $s,t\in \BZ$:
\begin{align*}
&  u_nv\in V_{s+t-n-1} \quad \text{ for } u\in V_s, v\in V_t;  \\
& [L(m),L(n)]=(m-n)L(m+n)+\frac{1}{12}(m^3-m)\delta_{m+n,0}c ;\\
& \frac{d}{dz}Y(v,z)=Y(L(-1)v,z);\\
& L(0)|_{V_n}=n,
\end{align*}
where $L(m)=\om_{ m+1}$, that is,
$Y(\om,z)=\sum_{n\in\BZ}L(n)z^{-n-2}$ and $n$ is regarded as a number in $\BF.$  Clearly, $D=L(-1)$  for a vertex operator
algebra.  If $v\in V_s$ we will call $s$ the degree of $v$ and write $\deg v=s.$

In the case when $\BF=\BC,$ the assumption $u_nv\in V_{s+t-n-1}$ in the definition
is a consequence of the other axioms.

A weak $V$-module $M$ is a vector space over $\BF$ equipped with a
linear map
\begin{align*}
&V\to \End(M)[[z, z^{-1}]]\\
&v\mapsto\displaystyle{
Y_M(v,z)=\sum_{n\in\BZ}v_nz^{-n-1}\quad(v_n\in\End( M))}
\end{align*}
which satisfies that for all $u, v\in V$, $w\in M,$ $l\in\BZ,$
\begin{equation*}
\begin{split}
&u_lw=0 \text{ for }l\gg0;\\
& Y_M(1 ,z)=Id_{M};\\
z^{-1}_0\delta
&\left(\frac{z_1-z_2}{z_0}\right)Y_M(u,z_1)Y_M(v,z_2)-z^{-1}_0\delta\left(\frac{z_2-z_1}{-z_0}\right)Y_M(v,z_2)Y_M(u,z_1)\\
&=z_2^{-1}\delta\left(\frac{z_1-z_0}{z_2}\right)Y_M(Y(u,z_0)v,z_2).\end{split}
\end{equation*}

We need the following Lemma (cf. \cite{L1}, \cite{DM}, \cite{LL}).
\begin{lem}\label{lm} If $M$ is a weak $V$-module generated by $w\in M.$ Then $M$ is spanned by $u_nw$ for $u\in V$ and $n\in\Z.$
In particular, if $M$ is irreducible, then we can take $w$ to be any nonzero vector in $M.$
\end{lem}

A  $V$-module is a weak $V$-module $M$ which carries a
$\BZ_+$-grading
\begin{equation*}\label{g2.22}
M=\oplus_{n\in\BZ_+}M(n)
\end{equation*}
satisfying
\begin{eqnarray*}\label{g2.23}
v_{m}M(n)\subseteq M(n+\deg v-m-1)
\end{eqnarray*}
for $v\in V.$ We will call $n$  the degree of $w\in M(n).$
If $M$ is irreducible then there exists $\lambda\in \BF$ such that $L(0)|_{M(n)}=\lambda+n$ for all $n\in\Z.$
A uniform grading shift gives an isomorphic module. As a result we always assume that $M(0)\ne 0.$

\begin{rem} The notion of $V$-module defined here is called admissible module in the case $\BF=\BC$ in \cite{DLM2} or $\BZ_{+}$-graded module in \cite{LL}. There is also a notion of ordinary $V$-module
in \cite{DLM2}. But the notion of ordinary module is not suitable in the current situation. So we simply call a $\Z_{+}$-module a module in this paper without any confusion.
\end{rem}

Let $M$ be a weak $V$-module. Then the
Jacobi identity is equivalent to the
associativity formula
\begin{eqnarray*}\label{ea2.15}
(z_{0}+z_{2})^kY_{M}(u,z_{0}+z_{2})Y_{M}(v,z_{2})w
=(z_{2}+z_{0})^{k}Y_M(Y(u,z_0)v,z_2)w
\end{eqnarray*}
where $w\in M$ and $k\in\BZ_{+}$ such that $z^{k}Y_{M}(u,z)w$ involves
only nonnegative integral powers of $z,$ and
commutator relation
\begin{eqnarray*}\label{g2.16}
[Y_{M}(u,z_{1}),Y_{M}(v,z_{2})]=\Res_{z_{0}}z_2^{-1}
\delta\left(\frac{z_1-z_0}{z_2}\right)Y_M(Y(u,z_0)v,z_2)
\end{eqnarray*}
whose component form is given by
\begin{eqnarray*}\label{g2.17}
[u_{s},v_{t}]=\sum_{i=0}^{\infty}
\left(\begin{array}{c}s\\i\end{array}\right)(u_{i}v)_{s+t-i}
\end{eqnarray*}
for all $u,v\in V$ and $s,t\in\Z$ (cf. \cite{FLM}, \cite{DL}, \cite{LL}).  One can also deduce the usual
Virasoro algebra axioms:
\begin{equation*}\label{g2.18}
[L(m),L(n)]=(m-n)L(m+n)+\frac{1}{12}(m^3-m)\delta_{m+n,0}c,
\end{equation*}
\begin{equation*}\label{g2.19}
\frac{d}{dz}Y_M(v,z)=Y_M(L(-1)v,z)
\end{equation*}
(cf. \cite{DLM1}, \cite{DLM2}) for $m,n\in\Z$ where
$Y_M(\om,z)=\sum_{n\in\Z}L(n)z^{-n-2}.$

Vertex operator algebra $V$ is called rational if the $V$-module category is semisimple. It is proved in \cite{DR1} that if $V$ is rational and $\frac{L(1)^n}{n!}$ is well defined on
$V$ for all $n\geq 0$ then $V$ has only finitely many inequivalent irreducible modules and the homogeneous subspaces of the irreducible modules are finite dimensional. But in the case
$\BF=\BC$ the assumption on $\frac{L(1)^n}{n!}$ is not necessary \cite{DLM2}. Removing this assumption for any field $\BF$ is one goal of this paper.

\section{Associative algebra $A_n(V)$  }
\def\theequation{2.\arabic{equation}}
\setcounter{equation}{0}

We fix a vertex operator algebra $V.$ The construction of $A_n(V)$ for any nonnegative integer $n$ is a suitable modification of that given in \cite{DLM3} in the case of complex field. The definition of $O_n(V)$ is more complicated in the current situation.

For any nonnegative integers $n,$ let $O_n(V)$ be the linear span of all $a\circ_{n,t}^s b$ and
$L(-1)a+L(0)a$
where for homogeneous $a\in V$ and $b\in V,$
$$
a\circ_{n,t}^s b=\Res_{z}Y(a,z)b\frac{(1+z)^{\deg
a+n+s}}{z^{2n+2+t}}$$
and $s,t\in\Z$ with $s\leq t.$ The notation
$a\circ_t^sb$ comes from \cite{DJ}. Also define $A_n(V)=V/O_{n}(V).$ For $a\in V$ we denote $a+O_n(V)$ by $[a].$
In the case $n=0,$ $A_0(V)$ is exactly the $A(V)$ investigated in \cite{DR1} (also see \cite{Z}).

As in \cite{DLM3} we define a  product $*_n$ on $V$ for $a$ and $b$ as
follows:
\begin{eqnarray*}
& & a*_nb=\sum_{m=0}^{n}(-1)^m{m+n\choose n}\Res_zY(a,z)\frac{(1+z)^{\deg
a+n}}{z^{n+m+1}}b\\
& &\ \ \
=\sum_{m=0}^n\sum_{i=0}^{\infty}(-1)^m
{m+n\choose n}{\wt a+n\choose i}a_{i-m-n-1}b.
\end{eqnarray*}

When $\BF=\BC,$ it was proved in
\cite{DLM3} that
$a\circ_{n,t}^sb$ is can be spanned by $a\circ_n b=a\circ_{n,0}^0 b$ for $a,b\in V.$
So we do not need $a\circ_{n,t}^sb$ in the definition of $O_n(V).$
But for an arbitrary field $\BF,$ the same result cannot be proved and we have to include $a\circ_{n,t}^sb$ in $O_n(V).$

Using the skew symmetry  $Y(a,z)b=
e^{zD}Y(b,-z)a$ and the same proof given in \cite{DLM3}  in the case $\BF=\BC$ we have the following result which will be useful in the proof
of the first main result in this paper.
\begin{lem}\label{l3.1}
For $a,b\in V,$

(1) $a*_nb \equiv \sum_{m=0}^n{m+n\choose
n}(-1)^n\Res_zY(b,z)a\frac{(1+z)^{\deg b+m-1}}{z^{1+m+n}}$
$\mathrm{(mod}\ O_n(V))$,

(2) $a*_nb-b*_na\equiv\Res_zY(a,z)b(1+z)^{\deg a -1}$ $\mathrm{(mod}\
O_n(V)).$
\end{lem}

 We use $o(a)$ for the operator $a_{\deg
a-1}$ on any weak $V$-module $M$ for a homogeneous $a\in V$,
and extend it to whole $V$ by linearly. It is clear from the definition of module
 that $o(a)M(n)\subset
M(n)$ for all
$n\in\BZ$ if $M$ is a $V$ module.
Here is our first main result which was obtained previously in \cite{DLM3} in the case that $\BF=\BC.$

\begin{thm}\label{the}
Let $V$ be a vertex operator algebra and $n$ a nonnegative integer.

(1) $A_n(V)$ is an associative algebra with product induced from $*_n$
and with identity $[1],$ central element $[\om].$

(2) If $n\geq m,$ $O_n(V)\subseteq O_m(V)$, and $A_m(V)$ is a quotient
of $A_n(V).$

(3) If $M$ is a weak $V$-module then
$$\Omega_n(M)=\{w\in M|u_iw=0, u\in V, i\geq \deg u+n\}$$
is an $A_n(V)$-module such that $[a]$ acts as $o(a)$ for homogeneous
$a\in V$.

(4) If $M=\oplus_{n\geq 0}M(n)$ is a  $V$-module, then $M(i)$ for each
$i=1,\ldots n$ is an $A_n(V)$-submodule of $\Omega_n(M).$
Moreover, if $M$ is irreducible then $\Omega_n(M)=\oplus_{i=0}^n M(i)$,
and $M(i)$ and $M(j)$ are
inequivalent simple $A_n(V)$-module if $i\neq j,$ and $M(i), M(j)$ are nonzero.

(5) If the operators $\frac{L(1)^i}{i!}$ make sense on $V$ for $i\geq
0,$ then the linear map
$$\phi:  a\mapsto e^{L(1)}(-1)^{L(0)}a$$
induces an anti-isomorphism $A_n(V)\to A_n(V)$.
\end{thm}

\begin{proof} (1) It is enough to  verify that $a*_n(b\circ_{n,t}^s c),
(b\circ_{n,t}^s c)*_na\in O_n(V)$ for $a,b,c\in V$ and $s,t\geq 0$ with $t\geq s$
as the rest of the proof is the same as in \cite{DLM3} in the case
$\BF=\BC.$

For homogeneous $a,b,c\in V,$
\begin{eqnarray*}
& &\ \ \ \ \ a*_n(b\circ_{n,t}^{s}c)\\
& & =\sum_{m=0}^n(-1)^m{m+n \choose n}{\rm Res}_{z_{1}}Y(a,z_{1})(b\circ_{n,t}^{s}c)\frac{(1+z_{1})^{\deg a+n}}{z_{1}^{n+m+1}}\\
& &\equiv\sum_{m=0}^n(-1)^m{m+n \choose n}{\rm Res}_{z_{1}}{\rm Res}_{z_{2}}\frac{(1+z_{1})^{\deg a+n}}{z_{1}^{n+m+1}}Y(a,z_{1})
\frac{(1+z_{2})^{\deg b+n+s}}{z_{2}^{2n+2+t}}Y(b,z_{2})c\\
& &\ \ \ -\sum_{m=0}^n(-1)^m{m+n \choose n}{\rm Res}_{z_{2}}{\rm Res}_{z_{1}}\frac{(1+z_{2})^{\deg b+n+s}}{z_{2}^{2n+2+t}}Y(b,z_{2})
\frac{(1+z_{1})^{\deg a+n}}{z_{1}^{n+m+1}}Y(a,z_{1})c\\
& &=\sum_{m=0}^n(-1)^m{m+n \choose n}{\rm Res}_{z_{0}}{\rm Res}_{z_{2}}
\frac{(1+z_{2}+z_{0})^{\deg a+n}}{(z_{2}+z_{0})^{n+m+1}}\frac{(1+z_{2})^{\deg b+n+s}}{(z_{2})^{2n+2+t}}
Y(Y(a,z_0)b,z_{2})c\\
& &=\sum_{m=0}^n(-1)^m{m+n \choose n}\sum_{j\geq 0}{\deg a+n\choose j}\sum_{i\geq0}{-n-m-1\choose i}
\\
& &\ \ \cdot {\rm Res}_{z_{0}}{\rm Res}_{z_{2}}\frac{(1+z_{2})^{\deg a+n-j+\deg b +n +s}}{z_{2}^{n+m+1+i+2n+2+t}}z_{0}^{i+j} Y(Y(a,z_0)b,z_{2})c\\
& &=\sum_{m=0}^n(-1)^m{m+n \choose n}\sum_{j\geq 0}{\deg a+n\choose j}\sum_{i\geq0}{-n-m-1\choose i}
(a_{i+j}b)\circ_{n,n+m+1+i+t}^{i+n+1+s}c\\
\end{eqnarray*}
which lies in $O_n(V)$ as $n+m+1+i+t\geq i+n+1+s$ for any $i\geq 0.$ By Lemma \ref{l3.1}, we have
\begin{eqnarray*}
& &\ \ \ \ a*_n(b\circ_{n,t}^{s}c)-(b\circ_{n,t}^{s}c)*_na\\
& &\equiv{\rm Res}_{z_1}Y(a,z_1)(b\circ_{n,t}^{s}c)(1+z_1)^{\deg a-1}\\
& &\equiv{\rm Res}_{z_{1}}{\rm Res}_{z_{2}}(1+z_{1})^{\deg a-1}
\frac{(1+z_{2})^{\deg b+n+s}}{z_{2}^{2n+2+t}}Y(a,z_{1})Y(b,z_{2})c\\
& &\ \ \  - {\rm Res}_{z_{2}}{\rm Res}_{z_{1}}
\frac{(1+z_{2})^{\deg b+n+s}}{z_{2}^{2n+2+t}}Y(b,z_{2})(1+z_{1})^{\deg a-1}Y(a,z_{1})c\\
& &={\rm Res}_{z_{0}}{\rm Res}_{z_{2}}(1+z_{2}+z_{0})^{\deg a-1}
\frac{(1+z_{2})^{\deg b+n+s}}{z_{2}^{2n+2+t}}Y(Y(a,z_{0})b,z_{2})c\\
& &=\sum\limits_{j\geq 0}{\deg a-1\choose j}{\rm Res}_{z_{0}}{\rm Res}_{z_{2}}z_{0}^j
\frac{(1+z_{2})^{\deg a+ \deg b +n+s-1-j}}{z_{2}^{2n+2+t}}Y(Y(a,z_{0})b,z_{2})c\\
& &=\sum\limits_{j\geq 0}{\deg a-1\choose j}(a_jb)\circ_{n,t}^sc
\end{eqnarray*}
is an element of $O_n(V).$ So $(b\circ_{n,t}^{s}c)*_na\in O_n(V).$

(2) From the definition,  $O_n(V)\subseteq O_m(V)$ when $n\geq m.$
The proof of $u*_nv=u*_{n-1}v$ modulo $O_{n-1}(V)$  in the case $\BF=\BC$ $\cite{DLM3}$ works here.

(3) Proving  that $\Omega_n(M)$ is an $A_n(V)$-module is equivalent to proving $o(a)=0$ on $\Omega_n(M)$ for $a\in O_n(V)$ and
$o(a*_nb)=o(a)o(b)$ for $a,b\in V.$
The proof of $o(a*_nb)=o(a)o(b)$  is the same as in \cite{DLM3}.
We now prove $o(a)=0$ for $a\in O_n(V).$ It is obvious that $o(L(-1)b+L(0)b)=0$ on $M$ for any $b\in V.$
The following computation gives an explicit expression of $o(a\circ_{n,t}^sb)$ on $\Omega_n(M):$
\begin{eqnarray*}
& &o(a\circ_{n,t}^sb) =o\left({\rm Res}_{z}\frac{(1+z)^{\deg a+n+s}}{z^{2n+t+2}}Y(a,z)b\right)\\
& &=\sum\limits_{j\geq 0}{ \deg a+n+s \choose j}(a_{j-2n-2-t}b)_{\deg a+\deg b-j+2n+t}\\
& &=\sum\limits_{j\geq 0}{\deg a+n+s\choose j}{\rm Res}_{z_{2}}{\rm Res}_{z_{1}-z_{2}}
\frac{z_{2}^{\deg a+\deg b-j+2n+t}}{(z_{1}-z_{2})^{2n+2+t-j}}Y(Y(a,z_{1}-z_{2})b,z_{2})\\
& &={\rm Res}_{z_{2}}{\rm Res}_{z_{1}-z_{2}}\frac{z_{1}^{\deg a+n+s}z_{2}^{\deg b+n-s+t}}{(z_{1}-z_{2})^{2n+2+t}}Y(Y(a,z_{1}-z_{2})b,z_{2})\\
& &={\rm Res}_{z_{1}}{\rm Res}_{z_{2}}\frac{z_{1}^{\deg a+n+s}z_{2}^{\deg b+n-s+t}}{(z_{1}-z_{2})^{2n+2+t}}Y(a,z_1)Y(b,z_2)\\
& &\ \ \ -{\rm Res}_{z_{2}}{\rm Res}_{z_{1}}\frac{z_{1}^{\deg a+n+s}z_{2}^{\deg b+n-s+t}}{(-z_{2}+z_{1})^{2n+2+t}}Y(b,z_2)Y(a,z_1)\\
& &=\sum_{i\geq 0}{-2n-t-2 \choose i}(-1)^ia_{\deg a+s-n-2-t+i}b_{\deg b+n-s+i+t}\\
& &\ \ \ -\sum_{i\geq 0}{-2n-t-2 \choose i}(-1)^{2n+2+t+i}b_{\deg b-n-s-2-i}a_{\deg a+n+s+i}.
\end{eqnarray*}
Since $a_{\deg a+n+s+i}=b_{\deg b+n-s+i+t}=0$ on $\Omega_n(M)$ for $t\geq s\geq 0$ and $i\geq 0,$
 $o(a\circ_{n,t}^s b)=0$ on $\Omega_n(M).$

(4) From the definitions of module and $\Omega_n(V),$ we see that  $M(i)\subset \Omega_n(M)$ if $i\leq n.$
We need to show that $\Omega_n(M)\cap M(i)=0$ if $i>n$, when $M$ is a simple $V$-module.
Assume that $\Omega_n(M)\cap M(i)$ is not zero for some $i>n.$ We
take a nonzero vector $w$ in $\Omega_n(M)\cap M(i).$
Then $M=\Span\{u_{\deg u+p}w\mid u\in V,p\in\Z, p<n\}$ by Lemma \ref{lm}.
This implies that $M(0)=0, $ a contradiction. So if $M$ is simple, $\Omega_n(M)=\oplus_{i=0}^nM(i).$

Let $M$ be an irreducible module. We now prove that each $M(i)$ is a simple $A_n(V)$-module for $i\leq n.$ Note that $M=\Span\{a_nw\mid a\in V, n\in \Z\}$, where $w$ is any fixed nonzero vector in $M(i).$
Note that  $a_nw\in M(\deg a-n-1+i)$ for any homogeneous $a\in V,$ we have $M(j)$ is spanned by $a_{\deg -1+i-j}w$ for $a\in V$ and $j\in\Z.$ In particular,  $M(i)$ is a simple $A_n(V)$-module.
The inequivalence of $M(i)$ and $M(j)$ in the case $\BF=0$ is trivial as $L(0)$ has different eigenvalues on $M(i)$ and $M(j).$ But for an arbitrary field
$\BF$ we have to find a different proof. Without loss, we can assume $j>i.$ Pick a nonzero  $u\in M(j)$ and any $v\in M(i).$ From the discussion above, there exists $a\in V$ such that $a_{\deg a-1+j}u\in M(0)$ is nonzero. Clearly, $a_{\deg a-1+j}v=0.$ Similarly there exists $b\in V$ such that $b_{\deg b-1-j}a_{\deg a-1+j}u\in M(j)$ is nonzero. By \cite{L2}, \cite{DM}, \cite{LL}, \cite{DR2} there exists $c\in V$ such that
$b_{\deg b-1-j}a_{\deg a-1+j}u=o(c)u$ and $b_{\deg b-1-j}a_{\deg a-1+j}v=o(c)v.$ Note from \cite{LL} that $c$ only depends on $b$ and $j.$ This shows that $o(c)\ne 0$ on $M(j)$ and $o(c)=0$ on $M(i).$ Thus
$M(i)$ and $M(j)$ are inequivalent $A_n(V)$-modules.

(5) We need to establish $\phi(a*_nb)=\phi(b)*_n\phi(a)$ modulo $O_n(V)$ and
$\phi(a\circ_{n,t}^sb)\in O_n(V)$ for any $a,b\in V$ and $t\geq s\geq 0.$
The proof is similar to that of Proposition 3.2 of \cite{DJ}.
\end{proof}

From Theorem \ref{the},  $\Omega_n/\Omega_{n-1}$ is a functor
from the completely reducible  $V$-module category to the completely reducible $A_n(V)$-module category whose irreducible components
can not factor through $A_{n-1}(V)$. The
restriction of $\Omega_n/\Omega_{n-1}$ sends the simple object to simple object.

\section{From $A_n(V)$-modules to $V$-modules}

We have discussed in Section 3 how to obtain $A_n(V)$-modules from $V$-modules. In this section we will go the opposite direction.
That is, we will construct a $V$-module
$M_n(U)=\oplus_{m\geq 0}M_n(U)(m)$ of Verma type for any $A_n(V)$-module $U$ which cannot factor through $A_{n-1}(V),$
such that $M_n(U)(0)\neq 0,$ and $M_n(U)(n)=U.$
(If it can factor through $A_{n-1}(V),$ we can consider the same procedure for $A_{n-1}(V).$)
As in the case that $\BF=\BC,$
$M_n(U)$ has a unique maximal submodule $W_n(U)$ and $L_n(U)=M_n(U)/W_n(U)$ is the smallest module such that $L_n(U)(0)\neq 0$ and $L_n(U)(n)=U.$
These results are then used to study the properties of $A_n(V)$ and the rationality of  $V$.

Notice from \cite{B1} that $\BF[t,t^{-1}]$ is a vertex algebra such
that $\1=1$ and $$Y(f(t),z)g(t)=(e^{z\frac{t}{dt}}f(t))g(t)$$
for $f,g\in \BF[t,t^{-1}].$
 Then we know $D^{(i)} f(t)=\frac{{(\frac{d}{dt})}^i}{i!}f(t),$ and ${\cal D}\BF[t,t^{-1}]=\sum_{n\ne-1}\BF t^n$  from \cite{DR1},
 where ${\cal D}=\sum_{i \geq0}D^{(i)}.$
The tensor product ${\cal L}(V)={\BF}[t,t^{-1}]\otimes V$  is a vertex algebra (cf. \cite{FHL}, \cite{L2}).
Let $\widehat V={\cal L}(V)/{\cal D}{\cal L}(V)$ be the corresponding Lie algebra such that for $a,b\in V$ and $p,q\in\Z,$ $[a(p), b(q)]=\sum_{i=0}^{\infty}{p\choose i}(a_ib)(p+q-i)$
where $a(p)$ is the image of $t^p\otimes a$ in $\widehat V.$

We define the degree of $a(m)$ to be $\deg a-m-1$ and let  $\widehat V_n$ be the degree $n$ subspace of $\widehat V.$ Then  $\widehat V=\bigoplus_{m\in\BZ}\widehat V_m$ is a $\BZ$-graded Lie algebra.
In particular, $\widehat{V}_0$ is a Lie subalgebra.
Recall Lemma \ref{lm}. As in \cite{DLM3} we have an epimorphism of Lie algebras from
$\widehat V_0$ to $A_n(V)_{Lie}$ by sending $a(\deg a-1)$ to $a+O_n(V)$
where $A_n(V)_{Lie}$ is the Lie algebra structure on $A_n(V)$ induced from the associative algebra structure.

We are ready to construct a $V$-module $M_n(U)$ from an $A_n(V)$-module $U$ which can not factor through $A_{n-1}(V).$
Then $U$ can be regarded as a module for $A_n(V)_{Lie}.$ From the discussion above, we
can make $U$ a module for $\widehat V_0.$ Note that $P_n=\oplus_{p>n}\widehat V_{-p}\oplus\widehat V_0$ is a subalgebra of  $\widehat V.$ We extend $U$ to a $P_n$-module by
letting $\widehat V_{-p}$
act trivially.
Consider induced module $\bar{M}_n(U)=\ind_{P_n}^{\hat V}(U)=U(\hat V)\otimes_{U(P_n)} U. $
If we give $U$ degree $n$, the $\Z$-gradation of $\hat V$ lifts to
$\bar{M}_n(U)$ which becomes a  $\Z$-graded module for $\hat V.$
It is easy to see that $\bar{M}_n(U)(i)=U(\hat V)_{i-n}U.$
We define for $v\in V,$
$$Y_{\bar{M}_n(U)}(v,z)=\sum_{m\in\Z}v(m)z^{-m-1}.$$
Then $Y_{\bar{M}_n(U)}(v,z)$ satisfies all conditions of a weak $V$-module except the associativity
which does not hold on $\bar{M}_n(U)$ in general \cite{DLM3}.

Motivated by the associativity relation, we let $W$ be the subspace of $\bar{M}_n(U)$ spanned linearly by the
coefficients of
\begin{eqnarray*}\label{g6.3}
(z_{0}+z_{2})^{{\deg}a+n}Y(a,z_{0}+z_{2})Y(b,z_{2})u-(z_{2}+z_{0})^{{\deg}a+n}
Y(Y(a,z_{0})b,z_{2})u
\end{eqnarray*}
for any homogeneous $a\in V,b\in V,$ $u\in U$ as in \cite{DLM3}.
Set
$$ M_n(U)=\bar{M}_n(U)/U(\hat V)W.$$

Let $U^*=\Hom_{\BF}(U,\BF)$ and let $U_s$ be the subspace
of $\bar{M}_n(U)(n)$ spanned by  ``length'' $s$ vectors
 $$o_{p_1}(a_1)\cdots o_{p_s}(a_s)U$$
where $p_1\geq \cdots \geq p_s,$ $p_1+\cdots p_s=0,$ $p_i\ne 0,$ $p_s\geq -n,$
$a_i\in V$ and $o_j(a)=a(\deg a-1-j)$ for homogeneous $a\in V.$
The  PBW theorem gives
$\bar{M}_n(U)(n)=\sum_{s\geq 0}U_s$ with $U_0=U$ and $U_s\cap U_t=0$ if $s\ne t.$

For homogeneous $u\in V,$ $v\in V$ and $m,q,p\in{\mathbb Z}_{+}$,
define the product $\ast_{m,p}^{q}$  on $V$ \cite{DJ} as follows
$$
u\ast_{m,p}^{q}v=\sum\limits_{i=0}^{p}(-1)^{i}{m+q-p+i\choose
i}{\rm Res}_{z}\frac{(1+z)^{\deg u+m}}{z^{m+q-p+i+1}}Y(u,z)v.
$$
Then we extend $U^*$ from $U$ to $M_n(U)(n)$ inductively so that
\begin{eqnarray*}\label{def}
\<u',o_{p_1}(a_1)\cdots o_{p_s}(a_s)u\>
=\<u',o_{p_{1}+p_2}(a_1\ast_{m,m+p_1}^{m+p_1+p_2}a_2)o_{p_3}(a_3)\cdots o_{p_{s}}(a_{s})u\>
\end{eqnarray*}
where $m=n+\sum_{i=3}^sp_i.$ We further extend
$U^*$ to $\bar{M}_n(U)$ by letting $U^*$ annihilate $\oplus_{i\ne n}\bar{M}_n(U)(i).$

Set $$ J=\{v\in M_n(U)|\langle u',xv\rangle=0\ {\rm for\ all}\ u'\in U^{*},\ {\rm all}\ x\in U(\hat V)\}.$$
Using the exact proof in \cite{DLM3} we have the second main result in this paper:

\begin{thm}\label{tm2} (1) $\ M_n(U)$ is a $V$-module with $ M_n(U)(0)\ne 0,$  generated by $M(U)(n)=U$ and satisfies the following universal property:
for any weak $V$-module $M$ and any $A_n(V)$-morphism $\phi: U\to \Omega_n(M),$ there is a unique
morphism $\bar\phi:  M_n(U)\to M$ of weak  $V$-modules which extends $\phi.$

(2) The $\bar M_n(U)$ has a unique maximal graded $\widehat V$-submodule $J$ with the property that $J\cap U=0.$
Then $L_n(U)=\bar M_n(U)/J$ is a $V$-module generated by $\Omega_n/\Omega_{n-1}(L_n(U))\cong U.$
Moreover, $U$ is simple if and only if $L_n(U)$ is irreducible,
and $U\mapsto L_n(U)$ gives a bijection between simple $A_n(V)$-modules which are not $A_{n-1}(V)$-modules and irreducible  $V$-modules.

(3) The $\bar J=J/U(\widehat V)W$ is the unique maximal
submodule of $M_n(U)$ such  that $\bar J\cap M_n(U)(n)=0$
and $L_n(U)=M_n(U)/\bar J.$
\end{thm}

It is clear that $M_n(U)$ is a Verma type $V$-module generated by $U$ and
$L_n(U)$ is the minimal $V$-module generated by $U.$

The next theorem which is an analogue of Theorem 4.10 of \cite{DLM3}
for rational vertex operator algebras.
\begin{thm}\label{t8.1} Suppose that $V$ is a rational vertex operator
algebra. Then the
following hold:

(1) $A_n(V)$ is a finite-dimensional, semisimple associative algebra for $n\geq 0.$

(2) If $M=\oplus_{n\geq 0}M(n)$ is an irreducible  $V$-module with $M(0)\ne 0,$
then $\dim M(m)<\infty$ for all $m.$

(3)  Let $M^i$ for $i=0,...,p$ be the inequivalent irreducible $V$-modules. Then
$$A_n(V)=\bigoplus_{i=0}^p\bigoplus_{m\leq n}\End M^i(m).$$
\end{thm}

\begin{proof} (1) It is good enough to show that any $A_n(V)$-module $U$ is completely reducible. We prove by induction on $n.$ If $n=0$ the result
is true from \cite{DR1}. Let $U$ be an $A_n(V)$-module. If $U$ is also an $A_{n-1}(V)$-module, then $U$ is a completely reducible $A_{n-1}(V)$-module by induction assumption.
Otherwise $U$ is an $A_n(V)$-module which can not factor through $A_{n-1}(V).$ From Theorem \ref{tm2} we have a $V$-module $M_n(U)=L_n(U)$ such that $M_n(U)(n)=U.$
Since $M_n(U)$ is a completely reducible $V$-module, we immediately see that $U$ is a completely reducible $A_n(V)$-module.

(2) follows from (1) as $A_n(V)$ is finite dimensional.

(3) We have already known from \cite{DR1} that $V$ has only finitely many inequivalent irreducible $V$-modules.   By (1) and Theorem \ref{the} we know that $M^i(m)$ for $i=0,...,p$ and $m\leq n$ form a complete list of inequivalent simple $A_n(V)$-modules. By Artin-Wedderburn Theorem,
 $A_n(V)=\bigoplus_{i=0}^p\bigoplus_{m\leq n}\End M^i(m),$ as desired.
 \end{proof}

We remark  that Theorem \ref{t8.1} also holds if
$\BF$ is a finite field. This is because
the Artin-Wedderburn Theorem is valid for semisimple associative
algebras over finite fields.

In the case of complex field, one can get a stronger result \cite{DLM3}: if $A_n(V)$ is finite dimensional semisimple associative algebra for all $n\geq 0$ then $V$ is rational. The key observation in \cite{DLM3} is that
$M(n)\ne 0$ if $n$ is sufficiently large for any irreducible module $M.$
\begin{lem}\label{kl} Let $V$ be a simple vertex operator algebra and $M=\oplus_{n\geq 0}M(n)$ be a $V$-module with $M(0)\ne 0$ and $L(0)=\lambda$ on $M(0).$  Then  $M(n)$ is nonzero for sufficiently large $n.$
\end{lem}
\begin{proof} Let $\ch \BF=p.$ Assume that $M(n)=0$ for some $n>0.$ Let $n=mp+r$ for some $m\geq 0$ and $r\in\{1,...,p-1\}.$ We claim that $L(-t)M(0)=0$ for all $0<t\leq mp.$ First we assume that $r=0.$ Then
$L(s)L(-n)M(0)=L(s)L(-mp)M(0)=sL(-mp+s)M(0)=0$ for $s=1,...,p-1.$ Also $L(1)L(-mp+p-1)M(0)=2L(-mp+p)M(0)=0.$ Continuing in this way gives the result. If $r>0$ then $L(r)L(-n)M(0)=2rL(-mp)M(0)=0.$
Since $2r\ne 0$ we see that $L(-pm)M(0)=0.$ Consequently, $L(-t)M(0)=0$ for all $0<t\leq mp.$

Assume that there are infinitely many $n_i$ such that $n_i<n_{i+1}$ for all $i$ and $M(n_i)=0.$  From the argument above we see that $L(-t)M(0)=0$ for all $t>0.$ Then $L(1)L(-1)M(0)=2L(0)M(0)=0.$ This forces $L(0)=0$
on $M(0).$ As a result, $Y(\omega, z)M(0)=0.$ Since $V$ is a simple vertex operator algebra, we conclude that $Y(u,z)M(0)=0$ for all $u\in V$ by using the exact argument give in Proposition 4.5.11 of \cite{LL}. In particular, $Y(\1,z)M(0)=0.$ This is a contradiction.  The proof is complete.
\end{proof}

The proof of Lemma \ref{kl} in the complex case \cite{DLM3} is easy as $L(-1)$ is an injective map from $M(n)$ to $M(n+1)$ if the kernel of $L(0)$ on $M(n)$ is $0.$

Finally we can have the following rationality result:
\begin{thm} Let $V$ be a simple vertex operator algebra. Then $V$ is rational if and only if $A_n(V)$ is a finite dimensional semisimple associative algebra for all $n\geq 0.$
\end{thm}
\begin{proof} By Theorem \ref{t8.1} we only need to show $V$ is rational if $A_n(V)$ are semisimple for all $n.$ So we need to prove any $V$-module $M$ is completely reducible. We first claim that
if $M=\oplus_{n\geq 0}M(n)$ is $V$-module generated by irreducible $A(V)$-module $M(0)=U$ then $M$ is irreducible.

Let $W$ be the maximal proper submodule of $M.$ If $W\ne 0$ then $W=\oplus_{n\geq m}W(n)$ for some $m>0$ and $W(m)\ne 0$ where $W(n)=W\cap M(n)$ for all $n.$ Recall that $L(U)=M/W$ is the irreducible $V$-module generated by $U.$ By Lemma \ref{kl} there is a large $n$ such that $W(n)\ne 0$ and
$L(U)(n)=M(n)/W(n)\ne 0.$ Since $A_n(V)$ is semisimple, $M(n)=X\oplus W(n)$ where $X$ is an $A_n(V)$-sumbodule isomorphic to $M(n)/W(n).$  Now let $P$ be the $V$-submodule generated by $X.$ Then
$P$ is a quotient of $M_n(X)$ by Theorem \ref{tm2} and $L_n(X)$ is the irreducible quotient of $P$ such that $L_n(X)(0)\ne 0.$ In particular, $P(n)=M_n(X)(n)=X.$ As both $P(0)$ and $L_n(X)(0)$ are irreducible $A(V)$-modules, we see that $P(0)=L_n(X)(0)=M(0)=U.$ Thus, $M$ is contained in $P$ as $M$ is generated by $U.$ This implies that $W(n)=0,$ a contradiction.

Now we prove any $V$-module $M=\oplus_{n\geq 0}M(n)$ with $M(0)\ne 0$ is completely reducible. Since  $M(0)$ is a direct sum of irreducible $A(V)$-modules,  the
$V$-submodule $M^0$ of $M$ generated by $M(0)$  is completely reducible. Decompose $M(1)=M^0(1)\oplus Y$ as $A_1(V)$-module. Then $Y$ is a completely reducible $A_1(V)$-module.

 Let $Z$ be an irreducible
$A_1(V)$-submodule of $Y.$ We claim that the $Z$ is an $A(V)$-module, or equivalently, $u_{\deg u+n}Z=0$ for all $u\in V$ and $n\geq 0.$ Assume that  $Z$ is not an $A(V)$-module. Let $Q$ be a $V$-module generated
by $Z.$ Then $Q\cap M(0)\ne 0$ and $Q(1)=Z$ where $Q(n)=M(n)\cap Q.$  From Theorem \ref{tm2}, $L_1(Z)$ is an irreducible $V$-module such that $L_1(Z)(1)=Z$ and $L_1(Z)(0)\ne 0.$ Clearly, $Q\cap M^0(1)=0.$ 
On the other hand, $L_1(Z)$ is a quotient of $Q$ and $U=L_1(Z)(0)\ne 0$ is an irreducible $A(V)$-submodule of $Q(0)=Q\cap M(0).$ Note that the $V$-submodule $Q^1$ generated by $U$ is contained in $M^0$ and  irreducible.
Also, $Q^1(1)\subset M^0(1)\cap Z=0.$ Since $L_1(Z)$ and $Q^1$ are inequivalent irreducible $V$-modules which have the same top level $U,$ we have a contradiction.  As a result, $Y$ is an $A(V)$-module. 

Let $M^1$ be the $V$-submodule of $M$ generated by $Y.$ Then $M^0\oplus M^1$ is a completely reducible submodule of $M.$ Continuing in this way shows
that $M$ is completely reducible, as desired.
\end{proof}




\begin{thebibliography}{ABCDE}

\bibitem[B1]{B1}R. E. Borcherds, Vertex algebras, Kac-Moody algebras, and the Monster,
{\em Proc. Natl. Acad. Sci. USA} {\bf 83} (1986), 3068-3071.

\bibitem[B2]{B2} R. E. Borcherds, Modular moonshine. III, {\em Duke Math. J.} {\bf 93} (1998), 129-154.

\bibitem[BR]{BR} R. E. Borcherds and A. Ryba,  Modular moonshine. II, {\em Duke Math. J.} {\bf 33} (1996), 435-59.



\bibitem[DG1]{DG1}C. Dong and R. L. Griess, Integral forms in vertex
operator algebras, {\bf 365} {\em J. Algebra} {\bf 365} (2012),184-198.

\bibitem[DG2]{DG2} C. Dong and R. L. Griess, Lattice-integrality of certain group-invariant integral forms in vertex operator algebras, {\em J. Algebra}, to appear.


\bibitem[DJ]{DJ} C. Dong and C. Jiang, Bimodules associated to
vertex operator algebra, {\em Math. Z.} {\bf 289} (2008), 799-826.



\bibitem[DL]{DL} C. Dong and J. Lepowsky, Generalized Vertex Algebras and Relative Vertex Operators, {\em Progress in Math.} Vol.
112,Birkh\"{a}user, Boston 1993.

\bibitem[DLM1]{DLM1} C. Dong, H. Li and G. Mason,
Regularity of rational vertex operator algebras, {\em Adv. Math.} {\bf 132} (1997), 148-166.

\bibitem[DLM2]{DLM2} C. Dong, H. Li and G. Mason,
Twisted representations of vertex operator algebras, {\em Math. Ann.}
{\bf  310} (1998), 571--600.

\bibitem[DLM3]{DLM3} C. Dong, H. Li and G. Mason,
Vertex operator algebras and associative algebras, {\em  J. Algebra}
{\bf 206} (1998), 67-96.

\bibitem[DLM4]{DLM4} C. Dong, H. Li and G. Mason, Twisted representations of vertex operator algebras and associative algebras,
{\em Int. Math. Res. Not.} {\bf 8} (1998), 389-397.

\bibitem[DM]{DM} C. Dong and G. Mason,  On quantum Galois theory, {\em Duke Math. J.} {\bf 86} (1997), 305-321.

\bibitem[DR1]{DR1}  C. Dong and L. Ren, Representations of vertex operator algebras over an arbitrary field,
{\em J. Algebra} {\bf  408} (2014), 497-516.

\bibitem[DR2]{DR2} C. Dong and L. Ren, Vertex operator algebras
associated to the Virasoro algebra over an arbitrary field,  {\bf 368} {\em Trans. AMS.} (2016), 5177-5196.

\bibitem[DZ]{DZ} C. Dong and Z. Zhao, Twisted representations of
vertex operator superalgebras, {\em Comm. Contemp. Math.} {\bf 8} (2006),101-122.

\bibitem[FLM]{FLM} I. B. Frenkel, J. Lepowsky and A. Meurman,
Vertex Operator Algebras and the Monster, {\em Pure and Applied Math.,} Vol. {\bf 134}, Academic Press, 1988.

\bibitem[FHL]{FHL}I. Frenkel, Y. Huang and J. Lepowsky, On axiomatic
approaches to vertx operator algebras and modules, {\em Mem. AMS} {\bf104}, 1993.


\bibitem[GL]{GL} R. L. Griess Jr. and C. Lam, Groups of Lie type, vertex algebras, and modular moonshine£¬
 {\em Int. Math. Res. Not. IMRN} {\bf 2015} (2015), 10716¨C10755.



\bibitem[KW]{KW}V. Kac and W. Wang, Vertex operator superalgebras and
representations, {\em Contemporary Math.} {\bf Vol. 175} (1994),161-191.

\bibitem[LL]{LL} J. Lepowsky and H. Li, Introduction to Vertex
Operator Algebras, and Their Representations, {\em Progress in Math.}Vol. 227,Birkh\"{a}user, Boston 2004.

\bibitem[L1]{L1} H. Li, Representation theory and tensor product theory for modules for a vertex operator algebra, Ph.D. thesis, Rutgers University, 1994.

\bibitem[L2]{L2} H. Li, The regular representation, Zhu's $A(V)$-theoy
and the induced modules, \emph{J. Algebra} \textbf{238} (2001),159-193.

\bibitem[LM]{LM} H. Li and Q. Mu, Heisenberg VOAs over Fields of Prime Characteristic and Their Representations,  arXiv:1501.04314

\bibitem[MT]{MT} M. Miyamoto and K. Tanabe, Uniform product of
$A_{g,n}(V)$ for an orbifold model $V$ and $G$-twisted Zhu algebra,
{\em J. Algebra} {\bf 274} (2004), 80-96.

\bibitem[M]{M} Q, Mu,  Lattice vertex algebras over fields of prime characteristic, {\em  J. Algebra} {\bf 417} (2014), 39-51.

\bibitem[R1]{R1} R. McRae, On integral forms for vertex algebras associated with affine Lie algebras and lattices, {\em J. Pure Appl. Algebra}
{\bf  219}  (2015), 1236-1257.

\bibitem[R2]{R2} R. McRae,  Integral forms for tensor powers of the Virasoro vertex operator algebra $L(1/20)$ and their modules, {\em  J. Algebra} {\bf 431} (2015), 1-23.

\bibitem[Z]{Z} Y. Zhu, Modular invariance of characters of vertex
operator algebras, {\em J. Amer, Math. Soc.} {\bf 9} (1996), 237-302.
\end{thebibliography}
\end{document}